\documentclass{amsart}
\usepackage{amssymb,graphicx}
\usepackage{todonotes}

\newtheorem{theorem}{Theorem}

\newtheorem{lemma}[theorem]{Lemma}
\newtheorem{proposition}[theorem]{Proposition}

\newtheorem{conjecture}[theorem]{Conjecture}
\theoremstyle{remark}

\newcommand{\cgC}{\mathcal{C}}
\newcommand{\cgD}{\mathcal{D}}
\newcommand{\cgE}{\mathcal{E}}

\newcommand{\cgI}{\mathcal{I}}
\newcommand{\cgL}{\mathcal{L}}

\newcommand{\cgO}{\mathcal{O}}
\newcommand{\cgR}{\mathcal{R}}
\newcommand{\cgS}{\mathcal{S}}
\newcommand{\cgT}{\mathcal{T}}
\newcommand{\Inc}{\operatorname{Inc}}

\newcommand{\Min}{\operatorname{Min}}

\newcommand{\mmD}{\mathbb{D}}

\newcommand{\AFB}{\operatorname{AFB}}

\title[PLANAR POSETS]{%
 Planar Posets that are Accessible from Below\\
  Have Dimension at Most~$6$}

\author[BIR\'{O}]{Csaba Bir\'{o}}
\address{(Biro)
  Department of Mathematics\\
  University of Louisville\\
  Louisville, Kentucky 40292\\
  U.S.A.}
\email{csaba.biro@louisville.edu}

\author[BOSEK]{Bart\l{}omiej Bosek}
\address{(Bosek) Theoretical Computer Science Department\\
  Faculty of Mathematics and Computer Science\\
  Jagiellonian University\\
  Krak\'ow, Poland}
\email{bosek@tcs.uj.edu.pl}

\author[SMITH]{Heather C. Smith}
\address{(Smith) Department of Mathematics and Computer Science\\
  Davidson College\\
  Davidson, North Carolina 28035\\
  U.S.A.}
\email{hcsmith@davidson.edu}

\author[TROTTER]{William T. Trotter}
\address{(Trotter) School of Mathematics\\
  Georgia Institute of Technology\\
  Atlanta, Georgia 30332\\ 
  U.S.A.}
\email{trotter@math.gatech.edu}

\author[WANG]{Ruidong Wang}
\address{(Wang) Blizzard Entertainment\\
  Irvine, California\\
  U.S.A.}
\email{rwang49@math.gatech.edu}

\author[YOUNG]{Stephen J. Young}
\address{(Young) Pacific Northwest National Laboratory\\
  Richland, Washington 99352\\
  U.S.A.}
\email{stephen.young@pnnl.gov}
\thanks{\emph{PNNL Information Release:} PNNL-SA-144431}

\date{June 19, 2019}



\thanks{B. Bosek is supported by Polish National Science Center 
grant 2013/11/D/ST6/03100.}

\begin{document}

\begin{abstract}
Planar posets can have arbitrarily large dimension.  However, a 
planar poset of height~$h$ has dimension at most~$192h+96$, 
while a planar poset with~$t$ minimal elements has dimension at 
most~$2t+1$.  In particular, a planar poset with a unique
minimal element has dimension at most~$3$.  In this paper, we extend
this result by showing that a planar poset has dimension at most~$6$
if it has a plane diagram in which every minimal element is
accessible from below.
\end{abstract}

\maketitle

\section{Introduction}

A non-empty family $\cgR$ of linear extensions
of a poset $P$ is called a \textit{realizer} of $P$ when
$x\le y$ in $P$ if and only if $x\le y$ in $L$ for each $L\in\cgR$.
The \textit{dimension} of a poset $P$, as defined by Dushnik
and Miller in their seminal paper~\cite{bib:DusMil}, is the least
positive integer $d$ for which $P$ has a realizer $\cgR$ with
$|\cgR|=d$. 

In recent years, there has been considerable interest in bounding
the dimension of a poset in terms of graph theoretic properties
of its cover graph and its order diagram.  For example, the following
papers link the dimension of a poset with tree-width, forbidden minors,
sparsity and game coloring numbers:\quad\cite{bib:JMMTWW},
\cite{bib:MicWie}, \cite{bib:Walc}, \cite{bib:JoMiWi-1} and~\cite{bib:JoMiWi-2}.
The results presented here continue in this theme.

Recall that a poset $P$ is said to be \textit{planar} if its order
diagram (also called a Hasse diagram) can be drawn without edge crossings
in the plane.  As is well known, a planar poset has an order
diagram without edge crossings in which edges are straight line segments.
Nevertheless, we elect to consider order diagrams in which covering
edges can be piecewise linear, as this convention simplifies our illustrations.
Given a planar poset $P$, a drawing of the order diagram of $P$
using piecewise linear paths for edges such that there
are no edge crossings will simply be called a 
\textit{plane diagram} of $P$.

In discussing a plane diagram $\mmD$ for a poset $P$, we will assume, without loss of generality, that no two points of $P$ lies on the same horizontal or vertical line in the plane.
We will also 
discuss points in the plane which do not correspond to elements of $P$.  
In particular, the set of points in the plane which 
do not correspond to elements of
$P$ and do not lie on the piecewise linear covering edges 
in $\mmD$ is partitioned
into one or more simply connected regions.   In general, there can be arbitrarily
many bounded regions, however the boundaries of these
regions need not be simple closed curves.  
Among these regions, there
is always a unique, unbounded region which is usually referred to as
the \textit{exterior region}. 
 
Let $\mmD$ be a plane diagram for poset $P$, and let
$x$ be a minimal element of $P$.  We will say that
$x$ is \textit{accessible from below} when 
there is a positive number $\epsilon=\epsilon(x)$ so that any 
point $p$ in the plane which is distinct from $x$, on the 
vertical ray emanating downwards from $x$ and 
within distance $\epsilon$ from $x$ is in the exterior 
region.  In turn we say that a plane diagram $\mmD$ is 
\textit{accessible from below} if
if every minimal element of $P$ is accessible from below.

We find it convenient to abbreviate the phrase ``accessible from
below'' with the acronym $\AFB$, so we will say that a minimal
element $x$ is $\AFB$ in a diagram $\mmD$, and we will
refer to an $\AFB$-diagram.  We then say that a poset $P$ is 
an $\AFB$-poset if it has an $\AFB$-diagram. All $\AFB$-posets 
are planar, but there are planar posets which are not $\AFB$.
Also, an $\AFB$-poset can have many plane diagrams with some of them
$\AFB$-diagrams and others not.  We illustrate this situation
in Figure~\ref{fig:two_drawings} where we show two plane diagrams
of an $\AFB$-poset $P$.  The diagram on the right is an $\AFB$-diagram 
while the diagram on the left is not.
 
\begin{figure}\label{fig:two_drawings}
\begin{center}
\includegraphics[scale=.6]{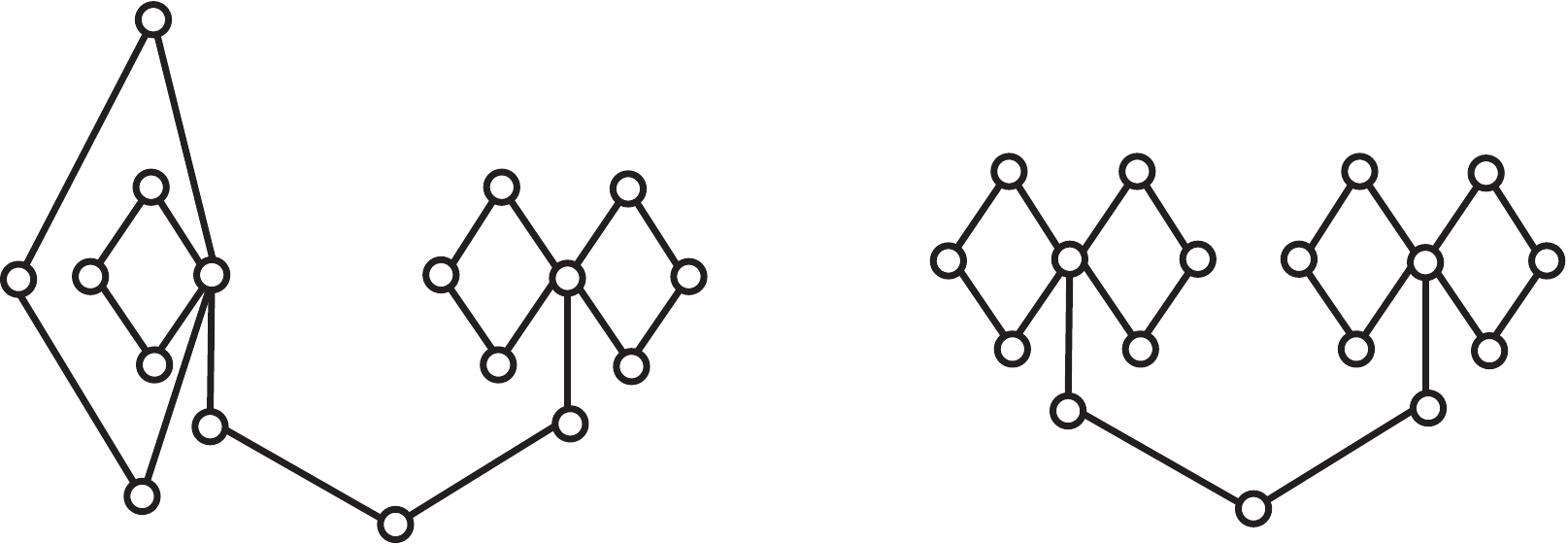}

\end{center}
\caption{The diagram on the right is an $\AFB$-Diagram}
\end{figure}

The principal goal of this paper is to prove the following
upper bound on the dimension of an $\AFB$-poset.

\begin{theorem}\label{thm:main}
If $P$ is an $\AFB$-poset, then $\dim(P)\le 6$.
\end{theorem}

The remainder of this paper is organized as follows.
Some background material necessary for the proof of 
Theorem~\ref{thm:main} is summarized in the next section, 
and the proof our main theorem is given in Section~\ref{sec:proof}.  We close 
in Section~\ref{sec:close} with some comments on the motivation for 
this line of research and connections with open problems.

We note that for every $d\ge 6$, it is an easy exercise to construct an 
$\AFB$-poset $P$ for which one cannot argue that $\dim(P)\le d$ by any
of the other known results for planar posets. Although we do not know if our upper bound is best possible, as
detailed in Section~\ref{sec:close}, a finite upper bound on the dimension of $\AFB$-posets is sufficient for our long range goals.

\section{Background Material}\label{sec:prelim}

We use (essentially) the same notation and 
terminology for working with dimension 
as has been employed by several authors
in recent papers, including: \cite{bib:FeTrWi},
\cite{bib:StrTro}, \cite{bib:JMMTWW}, \cite{bib:TroWan} and~\cite{bib:TrWaWa},
so our treatment will be concise.

Let $P$ be a poset with linear
extension $L$ and let $(x,y)\in \Inc(P)$.  We say $L$ \textit{reverses} $(x,y)$ when $x>y$ in $L$.
When $S\subset\Inc(P)$, we say that $L$ \textit{reverses} $S$ when $L$ reverses
every pair in $S$.  When $\cgR$ is a family of linear extensions of
$P$, we say $\cgR$ reverses $S$ when, for each $(x,y)\in S$, there
is some $L\in\cgR$ such that $L$ reverses $(x,y)$. Evidently,
the dimension of $P$ is just the minimum size of a non-empty family
of linear extensions which reverses $\Inc(P)$.

In the discussion to follow, we sometimes express a linear order  on a finite
set by writing $[u_1<u_2<\dots <u_r]$, for example.

A subset $S\subset\Inc(P)$ is \textit{reversible} when there is a
linear extension $L$ of $P$ which reverses $S$.  When $k\ge2$,
a sequence $\{(x_i,y_i):1\le i\le k\}$ of incomparable pairs in $P$
is called an \textit{alternating cycle} (of length~$k$) 
when $x_i\le y_{i+1}$ in $P$ for all $i\in [k]$, which should be 
interpreted \emph{cyclically}, i.e., we also intend that
$x_k \le y_1$ in $P$.  For the balance of the paper, we will use similar 
cyclic notation without further comment.

An alternating cycle is \textit{strict}
if for each $i\in[k]$, $x_i\le y_j$ in $P$ if and only if $j=i+1$, and further, the sets $\{x_1,x_2,\dots,x_k\}$ and
$\{y_1,y_2,\dots,y_k\}$ are $k$-element antichains.  

A poset has dimension~$1$ if and only if it is a chain.
When $P$ is not a chain, the dimension of $P$ is just
the least positive integer $d\ge2$ for which there is a covering
$\Inc(P)=S_1\cup S_2\cup\dots\cup S_d$ with $S_i$ reversible
for each $i\in [d]$,
The following elementary lemma of Trotter and 
Moore~\cite{bib:TroMoo} characterizing reversible sets
has become an important tool in dimension theory.

\begin{lemma}\label{lem:ac}
Let $P$ be a poset and let $S\subseteq\Inc(P)$.  Then the following
statements are equivalent.

\begin{enumerate}
\item $S$ is reversible.
\item There is no $k\ge2$ for which $S$ contains an alternating cycle
of length~$k$.
\item There is no $k\ge2$ for which $S$ contains a strict alternating cycle
of length~$k$.
\end{enumerate}
\end{lemma}

When $x<y$ in a poset $P$, we refer to a sequence $W[x,y]=
(u_1,u_2,\dots,u_r)$ of elements
of $P$ as a \textit{witnessing path} from $x$ to $y$ when
$u_1=x$, $u_r=y$ and $u_i$ is covered by $u_{i+1}$ in $P$ whenever
$1\le i<r$.  In general, there are many different witnessing paths from $x$ to $y$
and, in most instances, it will not matter which one is chosen.

When $\mmD$ is a plane diagram for a poset $P$ and
$x<y$ in $P$, we will take advantage of the fact that 
there is a uniquely determined ``left-most''
witnessing path from $x$ to $y$.  Analogously, there is a uniquely
determined ``right-most'' witnessing path from $x$ to $y$.
In Figure~\ref{fig:left_right_paths},  
we show a portion of a plane diagram where there are a total
of $12$ witnessing paths from $x$ to $y$.  The left-most path is
shown using dotted edges while the right-most path is shown
using bold face edges.

\begin{figure}
\begin{center}
\includegraphics[scale=.6]{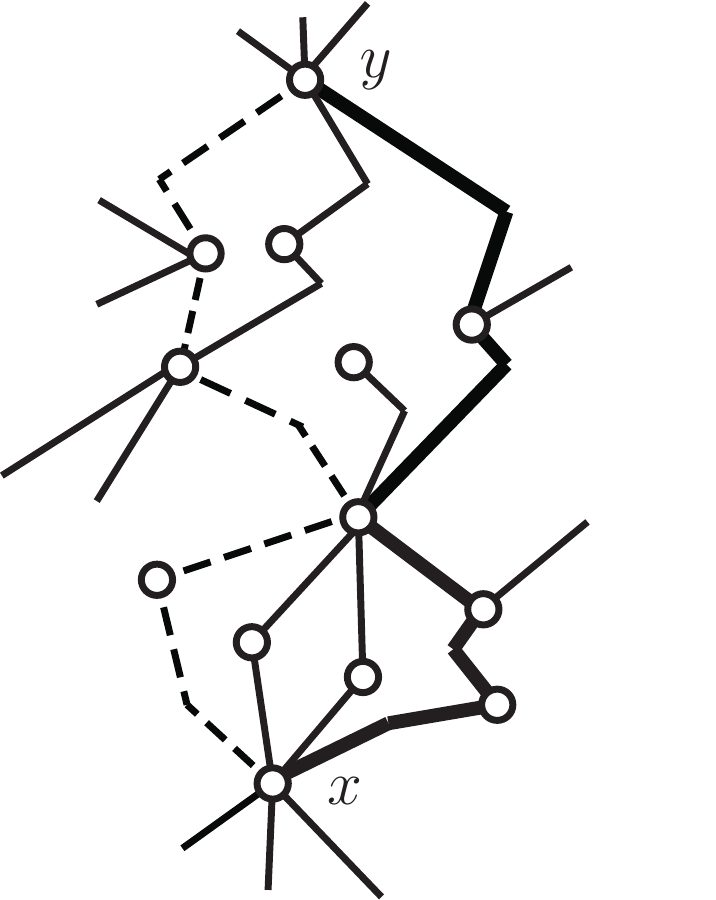}

\end{center}
\caption{Left-most and right-most witnessing paths}
\label{fig:left_right_paths}
\end{figure}

We can view a witnessing path $W[x,y]$ as a finite sequence of points 
of the poset $P$, but we can also view it as the simply connected
(and therefore infinite) set of points in 
the plane which belong to the covering edges in the path.
From the context of the discussion, it should 
be clear whether we intend a witnessing path to be simply a 
finite set of points from $P$ or an infinite set of points in the plane.
In the same spirit, we will splice witnessing paths together to form
simple closed curves in the plane.  These will always be infinite sets of points.

At a critical stage in our proof, we will discuss a simple closed 
curve $\cgE$ such that the minimal elements of $P$ are on $\cgE$, while
all other elements of $P$ are in the interior of $\cgE$. 

\subsection{Planar Posets with a Zero}

When a poset has a unique minimal element, that element is usually
referred to as a ``zero.''  Dually, if a poset has a unique maximal
element, then it is called a ``one.''
We state formally the theorem of Trotter and Moore~\cite{bib:TroMoo},
and give a short synopsis of a more modern proof given
in~\cite{bib:TroWan}, as these details will be important in proving
our main theorem.

\begin{theorem}\label{thm:planar-0}
If $P$ is a planar poset and $P$ has a zero, then
$\dim(P)\le 3$.
\end{theorem}
Since a poset and its dual have the same dimension, we also know that
a planar poset with a one has dimension at most~$3$.

Given a plane diagram $\mathbb{D}$ for a poset $P$ with a zero, let $L_1$ be the linear extension of $P$ obtrained from a depth-first search using a local left-to-right preference rule. Similarly, let $L_2$ be another linear extension of $P$ which is also obtained via a depth-first search, but with a right-to-left preference.  As noted in~\cite{bib:TroWan}, for every
$(x,y)\in\Inc(P)$, exactly one of the following four statements applies:

\begin{enumerate}
\item $x$ is right of $y$ (i.e. $x>y$ in $L_1$ and $x<y$ in $L_2$).
\item $x$ is left of $y$ (i.e. $x<y$ in $L_1$ and $x>y$ in $L_2$).
\item $x$ is outside $y$ (i.e. $x<y$ in both $L_1$ and$L_2$).
\item $x$ is inside $y$ (i.e. $x>y$ in both $L_1$ and$L_2$).
\end{enumerate}
In Figure~\ref{fig:planar-0}, we show a plane diagram $\mmD$ for a poset
$P$ with a zero in which  $10$ is right of~$5$, $7$ is left of~$9$,
$14$ is outside~$6$ and $10$ is inside~$13$.

\begin{figure}
\begin{center}
\includegraphics[scale=.6]{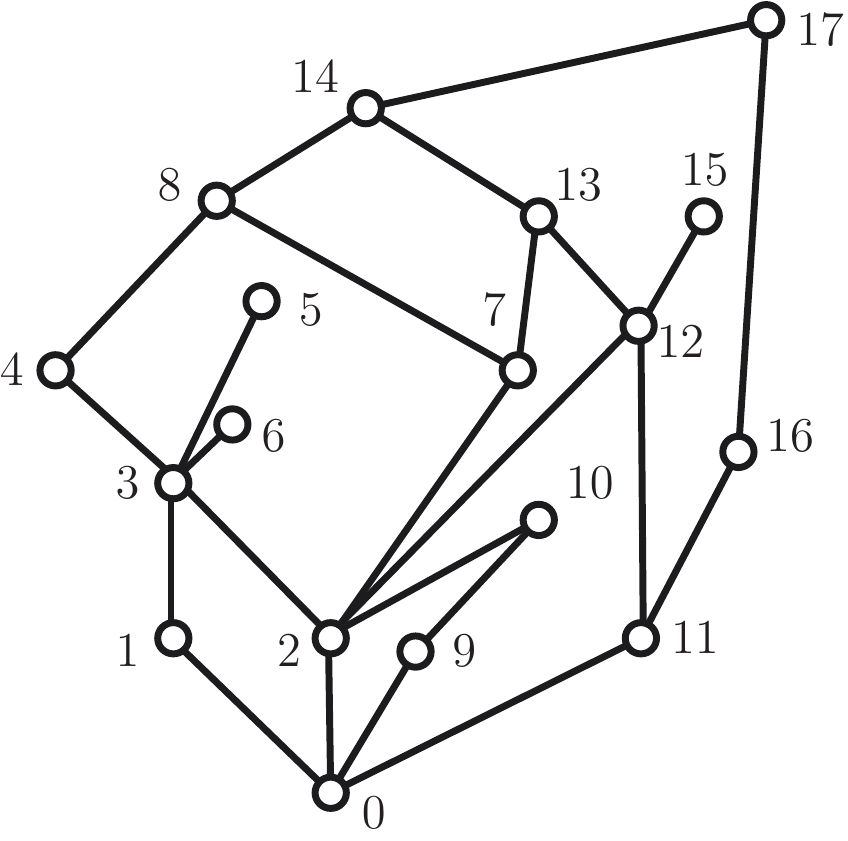}

\end{center}
\caption{A planar poset with a zero}
\label{fig:planar-0}
\end{figure}

Accordingly, it is natural to partition $\Inc(P)$ as
$\cgR\cup\cgL\cup\cgO\cup\cgI$, where $\cgR$ consists of all
pairs $(x,y)$ with $x$ right of $y$, etc.   The binary
relations $\cgR$ and $\cgL$ are complementary in the sense
that $x$ is left of $y$ if and only if $y$ is right of $x$.
Similarly, the binary relations $\cgI$ and $\cgO$ are complementary.
Also both $\cgL$ and $\cgR$ are transitive, e.g., if $x$ is
left of $y$ and $y$ is left of $z$, then $x$ is left of $z$.

The inequality $\dim(P)\le3$ is proved by showing that the
following sets are reversible: (1)~$\cgR\cup\cgO$, (2)~$\cgL\cup\cgO$,
and (3)~$\cgI$.  The arguments given in~\cite{bib:TroWan}
for the first two of these statements are constructive, as the desired
linear extensions are obtained via depth-first searches.  Note that the
labeling used in Figure~\ref{fig:planar-0} results from a depth-first
search using a local left-to-right preference rule.  This linear extension
illustrates that $\cgR\cup\cgO$ is reversible.  A depth-first search using
a local right-to-left preference rule will produce a linear extension
reversing all pairs in $\cgL\cup\cgO$.

To complete the proof, it is then
only necessary to show that $\cgI$ is reversible.
However, as pointed out in~\cite{bib:TroWan}, a somewhat more general
result holds: planar posets with $t$ minimal elements have dimension at most $2t+1$. 

Let $\mmD$ be a plane diagram for a planar poset $P$ (with no restriction
on the number of minimal elements of $P$).  When $z\in P$,
we let $U_P[z]$ consist of all elements $x\in P$ with $z\le x$ in
$P$.  The subposet $U_P[z]$ is planar and $z$ is a zero.  Accordingly,
we can classify the incomparable pairs in $U_P[z]$ using the
same four labels $\cgR$, $\cgL$, $\cgO$, and $\cgI$.
We will say that an incomparable pair $(x,y)$ 
in $P$ is an \textit{enclosed} pair when there is some $z\in P$ 
such that $x$ is inside $y$ in $U_P[z]$.  

For the benefit of readers who may be new to arguments using
alternating cycles, we give a proof for the following elementary
lemma.

\begin{lemma}\label{lem:enclosed}
Let $\mmD$ be a plane diagram for a poset $P$.  Then the
set $\cgS$ of all enclosed pairs in $P$ is reversible.
\end{lemma}

\begin{proof}
We argue by contradiction, supposing that $\cgS$ is not reversible.
Then by Lemma~\ref{lem:ac}, there is an integer $k\ge2$ and a strict alternating
cycle $\{(x_i,y_i):1\le i\le k\}$ of enclosed pairs.
For each $i\in[k]$, let $z_i$ be the unique element
of $P$ which is highest in the plane  with $x_i$ inside
$y_i$ in $U_P[z_i]$.   Then let $y'_i$ be the unique element
of $P$ which is lowest in the plane and satisfies both $y'_i\le y_i$ in
$P$ and $x_i$ is inside $y'_i$ in $U_P[z_i]$.  Then there
are two witnessing paths $W_1[z_i,y'_i]$ and $W_2[z_i,y'_i]$ which
form a simple closed curve $\cgC_i$ with $x_i$ in its interior.

Since $x_i\le y_{i+1}$ in $P$
and $x_i\parallel y_{i}$ in $P$, it follows that $y_{i+1}$
is also in the interior of $\cgC_i$.  Therefore $y_{i+1}$ is lower 
in the plane than $y'_i$. This is a 
contradiction since this statement cannot hold for all 
$i\in[k]$.
\end{proof}

In a dual manner, when $\mmD$ is a plane diagram for a poset $P$ and
$z\in P$, we define $D_P[z]$ as the subposet consisting of all
$x\in P$ with $x\le z$ in $P$.  The subposet $D_P[z]$ is planar, and 
the element $z$ is a one.  Now we can classify the incomparable
pairs in the subposet $D_P[z]$ using the same four labels but applied with the
obvious dual interpretation.  In general, if
$(x,y)$ is an incomparable pair in a subposet of the form
$D_P[z]$, then any of the four labels may be correct for the
pair $(x,y)$.  However, if $\mmD$ is an $\AFB$-diagram for a 
poset $P$, then two of the four
labels cannot be applicable. We state formally
the following nearly self-evident proposition for emphasis.
It does not hold for planar posets in general.

\begin{proposition}\label{pro:C2}
Let $\mmD$ be an $\AFB$-diagram for a poset $P$, and let
$z\in P$.  If $(x,y)$ is an incomparable pair in $D_P[z]$,
then either $x$ is left of $y$ in $D_P[z]$ or $x$ is right of 
$y$ in $D_P[z]$.  Furthermore, if $z'\in P$ and
$x,y\in D_P[z']$, then $x$ is left of $y$ in
$D_P[z]$ if and only if $x$ is left of $y$
in $D_P[z']$.
\end{proposition}

\section{Proof of our Main Theorem}\label{sec:proof}

In this section, we prove Theorem~\ref{thm:main}, i.e., we show
that if $P$ is an $\AFB$-poset, then $\dim(P)\le 6$.  
Our first step is to reduce the problem to a somewhat simpler one.

\medskip
\noindent
\textbf{Reduction.}\quad 
To show that the dimension of any $\AFB$-poset is at most $d$, it
suffices to show that whenever $\mmD$ is an $\AFB$-diagram for 
a poset $P$, the set of all incomparable pairs in $\Min(P)\times P$
can be covered by~$d-1$ reversible sets.  

\begin{proof}
Let $P$ be an $\AFB$-poset.
To show that $\dim(P)\le d$, we need to show that there
is a covering of the set of all incomparable pairs of $P$ by
$d$ reversible sets.  Let $\mmD$ be an $\AFB$-diagram
for $P$.  We will now show that $\mmD$ can be modified into 
an $\AFB$-diagram $\mmD'$ for a poset $P'$ such that:

\begin{enumerate}
\item $P'$ contains $P$ as a subposet.
\item If $(x,y)$ is an enclosed pair in $P$, then
$(x,y)$ is an enclosed pair in $P'$.
\item If $(x,y)$ is an incomparable pair of $P$ and
is not an enclosed pair in $P$, then 
 there is a 
minimal element $x'\in P'$ with $x'\le x$ in
$P'$ and $x'\parallel y$ in $P'$.  
\end{enumerate}

We will then let $S_0$ consist of all incomparable pairs
$(x,y)$ in $P$ such that $(x,y)$ is an enclosed pair in
$P'$.  The set $S_0$ is reversible by Lemma~\ref{lem:enclosed},
and it contains all enclosed pairs in $P$.
It remains to consider the incomparable pairs in $P$ which are not enclosed in $P'$.
In view of the third condition for $P'$, if the
incomparable pairs $(x',y)\in\Min(P')\times P'$ in $P'$ can be covered
by $d-1$ reversible sets, it follows that the set of all
incomparable pairs of $P$ can be covered by~$d$ reversible
sets.  So it only remains to explain how the poset $P'$ 
should be constructed from $P$.  

Let $S$ be the set of all incomparable pairs of $P$ which
are not enclosed pairs  in $P$ and do not belong to $\Min(P)\times P$.  
If $S=\emptyset$, simply take $\mmD'=\mmD$ and
$P'=P$.  So we may assume that $S\neq\emptyset$.  Let $r=|S|$ and
let $S=\{(x_i,y_i):1\le i \le r\}$ be an arbitrary labeling
of the pairs in $S$.

To initialize a recursive construction, we set $\mmD_0=\mmD$
and $P_0=P$.  We will now explain how to construct a sequence
$\{(\mmD_i,P_i):1\le i\le r\}$ such that for each
$i\in [r]$, $\mmD_i$ is an $\AFB$-diagram for
the poset $P_i$ where $M_i = \Min(P)$.  The construction will ensure that
$P_{i-1}$ is a subposet of $P_i$ and
$M_{i-1}$ is a subset of $M_i$ whenever
$1\le i\le r$.  Furthermore, for each $1\le j\le i\le r$, 
either $(x_j,y_j)$ is an enclosed pair in $P_i$ or there is a 
minimal element $x'_j\in M_j$ such that $x'_j\le x_j$ in 
$P_i$ and $x'_j\parallel y_j$
in $P_i$.  The $\AFB$-poset $P'$ is just $P_r$.

Now suppose that $0\le i < r$ and that we have defined the
$\AFB$-diagram $\mmD_i$ for $P_i$. We then consider the 
pair $(x,y)=(x_{i+1},y_{i+1})$.  
Let $x'$ be the uniquely determined element of $P$ which
is lowest point in the plane and satisfies 
$x'\le x$ in $P_i$ and $x'\parallel y$ in $P_i$.  
If $(x',y)$ is an enclosed pair in $P_i$, so is $(x,y)$.
Accordingly, if $x'\in\Min(P_i)$, or $(x',y)$ is an enclosed pair,
we simply take $\mmD_{i+1}=\mmD_i$ and $P_{i+1}=P_i$.

Now suppose $(x',y)$ is not an enclosed pair in $P_i$, and 
$x'$ is not a minimal element in $P_i$.
It follows that $x'$ covers one or
more elements in $P_i$.  We claim that $x'$ has a unique lower
cover.  Suppose to the contrary that $x'$ covers
distinct elements $u$ and $u'$ in $P_i$.  In view of our
choice of $x'$, we know $u,u'\in D_P[y]$.  However, since $(x',y)$ 
is not an enclosed pair in $P_i$ and $u\parallel u'$, the induced $\AFB$-diagram of the subposet of $P_i$ determined  by $\{u,u',x',y\}$ violates
Proposition~\ref{pro:C2}.  The contradiction shows that $x'$ covers
a unique point $u$ as claimed.  

Our choice of $x'$ implies that $u<y$ in $P_i$.
We consider the first edge of a witnessing path $W[u,y]$ and
the edge $ux'$.  The construction for $\mmD_{i+1}$ depends on which
of these two edges is left of the other at $u$.  In
Figure~\ref{fig:dirty_trick}, we show the first edge of 
$W[u,y]$ on the left, so the following discussion will
be reversed if the edge $ux$ is on the left.

Starting with $u$ and traveling down in the diagram, we always
proceed to the right-most lower cover until we reach a minimal element of $P$.  In Figure~\ref{fig:dirty_trick},
we suggest that this would result in the chain $(u>v>z>w)$ and
it should be clear how the following details should be modified
if the actual chain is of a different length.

Starting just above $u$ and headed downward, we insert new
points very close to the existing vertices---together with
intermediate vertices to ensure that the resulting figure
is a diagram.  This results in a new minimal element $x''$ with
$x''<x'\le x$ in $P_{i+1}$ and $x''\parallel y$ in $P_{i+1}$.
Again, we refer to Figure~\ref{fig:dirty_trick} as an example
for how these changes are to be made. Note that no new comparabilities are introduced among the points of $P_{i-1}$ with the addition of these new points.

\begin{figure}
\begin{center}
\includegraphics[scale=.6]{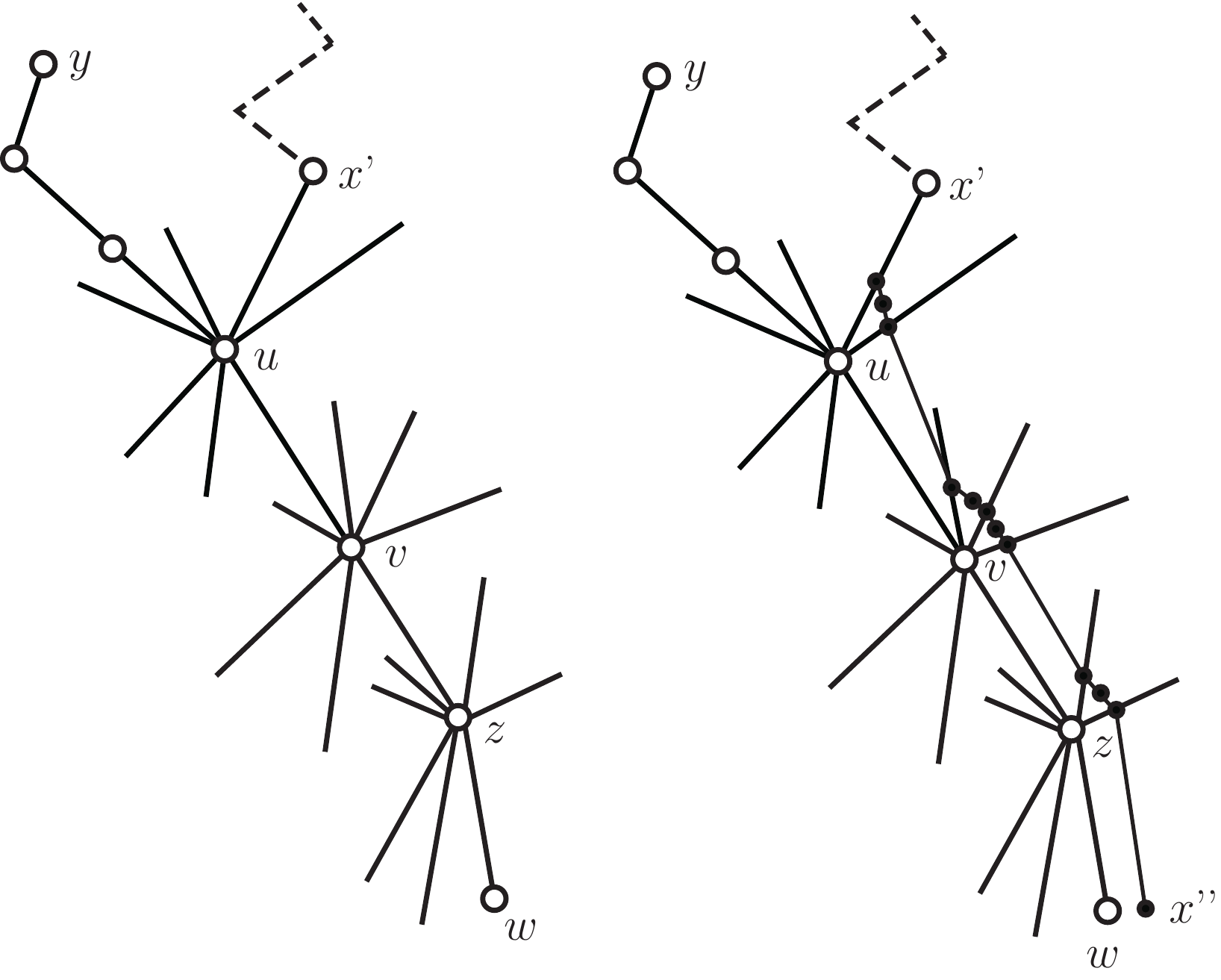}

\end{center}
\caption{The Construction for the Reduction}
\label{fig:dirty_trick}
\end{figure}

With this construction in hand, the proof for the reduction is complete.
\end{proof}

Given an $\AFB$-diagram for a poset $P$,
we know of no simple argument to show that the set of
incomparable pairs from $\Min(P)\times P$
can be covered with a bounded number of reversible sets, but
in time, we will show that $5$ are enough.  With the reduction,
this completes the proof that $\dim(P)\le6$ when $P$ is
an $\AFB$-poset.  However, to simplify the proof, we will first
prove a weaker result asserting that the set of incomparable pairs
from $\Min(P)\times P$ can be covered by~$7$ reversible sets.
The slight modification necessary to lower~$7$ to $5$ will
be presented later.

\begin{lemma}\label{lem:7_enough}
Let $\mmD$ be an $\AFB$-diagram for a poset $P$. Then
the set of all incomparable pairs of $P$ in $\Min(P)\times P$
can be covered by~$7$ reversible sets.
\end{lemma}

\begin{proof}
Clearly, it is enough to prove the lemma when $P$ is connected
and has at least two minimal elements.  We let $\cgS_0$ denote
the set of all incomparable pairs in $\Min(P)\times P$, and we 
abbreviate the set $\Min(P)$ as $M$.

Since $\mmD$ is an $\AFB$-diagram, it is easy to see that there
is a simple closed curve $\cgE$ in the plane satisfying the
following requirements:

\begin{enumerate}
\item All elements of $M$ are on $\cgE$.
\item All elements of $P-M$ are in the interior of $\cgE$.
\item If $x$ covered by $y$ in $P$, then all points of the plane
which are on the covering edge from $x$ to $y$ in the diagram
are in the interior of $\cgE$, except $x$ when $x\in M$.
\end{enumerate}

We illustrate such a curve in Figure~\ref{fig:envelope}
where we show $\cgE$ using dashed lines.  We find it
natural to refer to $\cgE$ as an \textit{envelope for $P$}.

\begin{figure}
\begin{center}
\includegraphics[scale=.6]{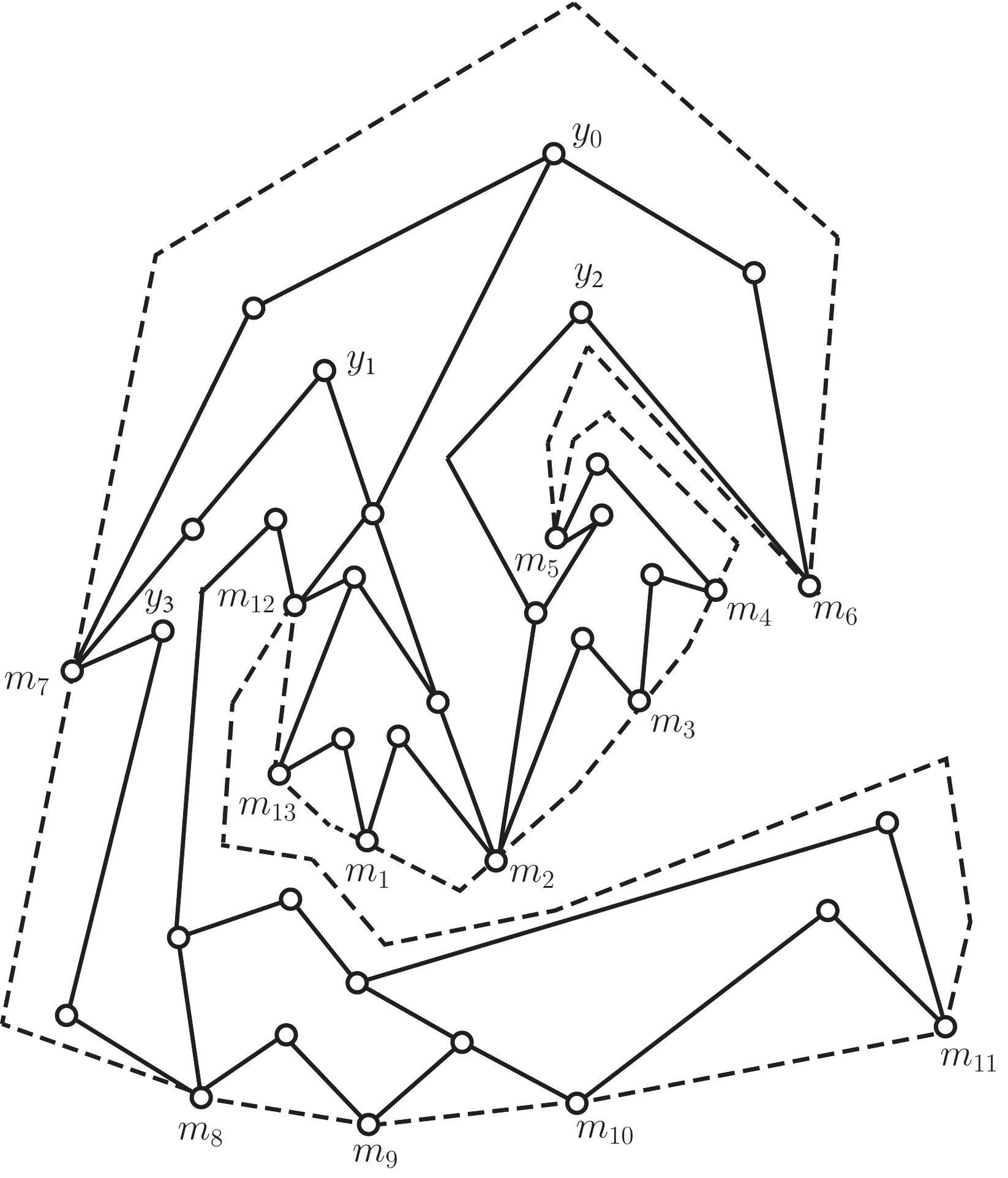}

\end{center}
\caption{An Envelope for an $\AFB$-poset}
\label{fig:envelope}
\end{figure}

Starting at an arbitrary minimal element $m_1$, we label the elements
of $M$ as they appear in a counter-clockwise traversal of $\cgE$ to
obtain a linear order \[L= [m_1<m_2<\dots<m_n]\] on $M$.  For
each element $y\in P$, we let $M[y]=M\cap D_P[y]$.

We will make repeated use of the following elementary proposition.
In fact, a stronger result holds, but this is the exact form we need.

\begin{proposition}\label{pro:Z}
Let $y\in P$, and let $m$ and $m'$ be distinct elements
of $M$ with $m,m'\in D_P[y]$.
Then let $Z$ be the subposet of $P$ consisting of all elements
$z\in P$ with $z\in D_P[y]$ such that $m,m'\in D_P[z]$.
Then the subposet $Z$ has a unique minimal element
which we will denote $z(y,m,m')$.
\end{proposition}

We will also make repeated use of a construction that produces
simple closed curves and regions in the plane.  Again,
let $y\in P$ and let $(m,m')$ be an \emph{ordered}
pair of distinct elements of $M[y]$.
Form a path $\cgE[m,m']$ by traversing the simple 
closed curve $\cgE$ in a counter-clockwise direction starting at $m$ and
stopping at $m'$.  Now $\cgE[m,m']$ and $\cgE[m',m]$ share
only $m$ and $m'$ as endpoints. Their union is
the entire curve $\cgE$.

Let $z=z(y,m,m')$.  We then take witnessing paths from
$m$ and $m'$ to $z$ using the following convention:
If $m$ is left of $m'$ in $D_P[z]$, then we take
$W[m,z]$ as the right-most witnessing path from $m$ to
$z$, while we take $W[m',z]$ as the left-most path from
$m'$ to $z$.  These conventions are reversed if $m'$
is left of $m$ in $D_P[z]$.

In either situation, the two witnessing paths $W[m,z]$ and
$W[m',z]$ together with the path $\cgE[m,m']$ form a simple
closed curve which we denote $\cgC(y,m,m')$.  Also,
we let $\cgR(y,m,m')$ denote the region in the
plane enclosed by $\cgC(y,m,m')$.  Note that $y$ is on
$\cgC(y,m,m')$ when $y=z(y,m,m')$.  However, when $y\neq
z(y,m,m')$, $y$ is in the exterior of $\cgR(y,m,m')$ when
$m$ is left of $m'$ in $D_P[y]$, and $y$ is in the
interior of $\cgR(y,m,m')$ when $m$ is right of $m'$ in
$D_P[y]$.

Now back to the argument for covering $\cgS_0$ with
$7$ reversible sets.  We will use the linear order $L$ to 
label the incomparable pairs in $\cgS_0$ using the 
following~$8$ labels:

\[
1A\quad 1B\quad 1C \quad 2A\quad 2B\quad 2C \quad 2D \quad 2E.
\]
The integer part of the label applied to a pair $(x,y)$ depends
only on $y$ while the letter in the label depends on both
$x$ and $y$. 

Let $y$ be an element of $P$.  Then the elements of
$M[y]$ are linearly ordered from left-to-right in 
$D_P[y]$.  We let $s(y)$ and $t(y)$ denote, respectively, the least 
element and the greatest element of $M[y]$ in this linear order.
Let $|M[y]|=r$ and let
$[u_1<u_2<\dots<u_r]$ be the left-to-right order on $M[y]$ in $D_P[y]$,
so that $s(y)=u_1$ and $t(y)=u_r$.

However, the elements of $M[y]$ are also linearly
ordered in $L$.  Now we let $a(y)$ and $b(y)$ denote,
respectively, the least element and the greatest element
of $M[y]$ in $L$.  Since the
envelope $\cgE$ is traversed in a counter-clockwise manner,
it is easy to see that $y$ can be characterized as one
of two types, since exactly one of the following two statements holds
for $y$:

\medskip
\noindent
\textbf{Type 1.}\quad $u_1<u_2<\dots<u_r$ in $L$.

\medskip
\noindent
\textbf{Type 2.}\quad
There is an integer $j$ with $1< j\le r$ such that:\\
$u_j<u_{j+1}<\dots<u_r<u_1<u_2<\dots<u_{j-1}$ in $L$.

\medskip
We note that an element $y\in P$ is Type~$1$ when $|M[y]|=1$.
In general, when
$y$ is Type~$1$, $a(y)=s(y)\le t(y)=b(y)$ in $L$.  When
$y$ is Type~$2$, $a(y) \le t(y)< s(y)\le b(y)$ in $L$.
Also, we observe that either $a(y)=b(y)$ or
$a(y)$ is left of $b(y)$ in $D_P[y]$ when
$y$ is Type~$1$.  However, $a(y)$ is right of $b(y)$ in
$D_P[y]$ when $y$ is Type~$2$.

Now let $(x,y)$ be a pair in $\cgS_0$.   If $y$ is Type~$1$,
we will say that $(x,y)$ is 
Type~$1A$ if $x<a(y)$ in $L$; Type~$1B$ if $a(y)<x<
b(y)$ in $L$; and Type~$1C$ if $x>b(y)$ in $L$.
In Figure~\ref{fig:envelope}, the elements $y_2$ and $y_3$
are Type~$1$.  The pairs $(m_1,y_2)$ and $(m_5,y_3)$ are
Type~$1A$;  the pairs $(m_3,y_2)$ and $(m_5,y_2)$ are Type~$1B$;
and the pairs $(m_8,y_2)$ and $(m_{12},y_3)$ are Type~$1C$.

When $y$ is Type~$2$, we say
the pair $(x,y)$ is Type~$2A$ if $x<a(y)$ in $L$;
Type~$2B$ if $a(y)<x<t(y)$ in $L$; Type~$2C$ if $t(y)<x<s(y)$ in $L$;
Type~$2D$ if $s(y)<x<b(y)$ in $L$; and Type~$2E$ if $x>b(y)$ in $L$.
In Figure~\ref{fig:envelope},
the elements $y_0$ and $y_1$ are Type~$2$.  Now $(m_1,y_0)$ and
$(m_1,y_1)$ are Type~$2A$; $(m_4,y_0)$ is Type~$2B$;
$(m_6,y_1)$ is Type~$2C$; $(m_{8},y_0)$ and $(m_{10}, y_1)$ are
Type~$2D$; and $(m_{13},y_0)$ and $(m_{13},y_1)$ are Type~$2E$.

We then define a covering of $\cgS_0$ by six sets
defined as follows:

\begin{enumerate}
\item $\cgS_1$ consists of all Type~$1A$ and~$2A$ pairs.
\item $\cgS_2$ consists of all Type~$1C$ and~$2E$ pairs.
\item $\cgS_3$ consists of all Type~$1B$ pairs.
\item $\cgS_4$ consists of all Type~$2B$ pairs.
\item $\cgS_5$ consists of all Type~$2D$ pairs.
\item $\cgS_6$ consists of all Type~$2C$ pairs. 
\end{enumerate}

We pause to examine the $\AFB$-poset shown in Figure~\ref{fig:top}
just to understand that there are obstacles to overcome
in covering $\cgS_0$ by a small number of reversible sets.
Referring to Figure~\ref{fig:top}, the set $\cgS_1\cup\cgS_2$
need not be reversible since $(x_4,y_4)$ is Type~$1A$ and
$(x_5,y_5)$ is Type~$1C$, but together these form a strict alternating cycle.
Also, $(x_1,y_1)$ and $(x_2,y_2)$ are Type~$2C$ while
$(x_3,y_3)$ is Type~$2B$.  No reversible set can contain
any two of these three pairs so $S_4 \cup S_6$ is not reversible and neither is $S_6$.

\begin{figure}
\begin{center}
\includegraphics[scale=.6]{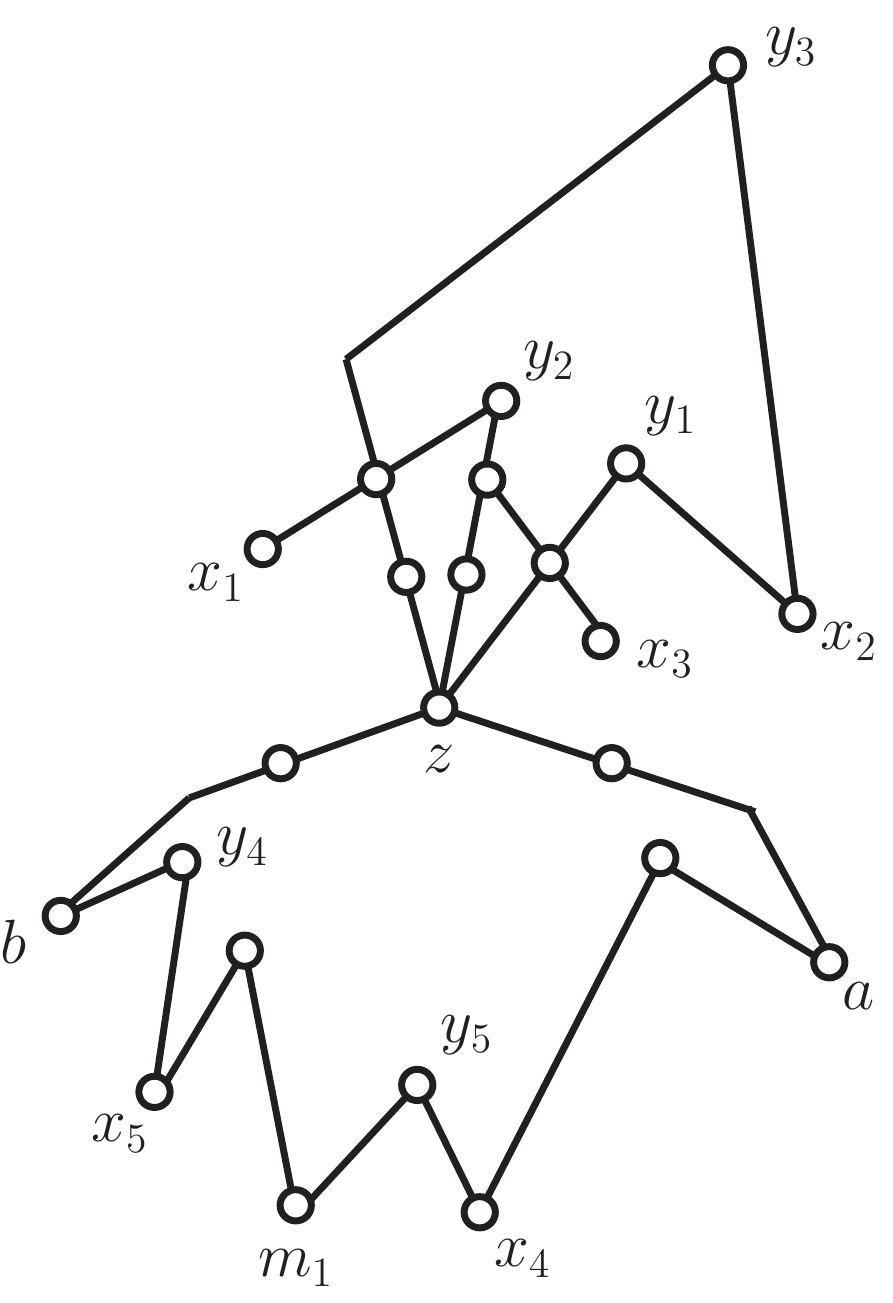}

\end{center}
\caption{Challenges in Reversing Pairs in $\cgS_0$}
\label{fig:top}
\end{figure}

Despite these challenges,
the proof of Lemma~\ref{lem:7_enough} and the proof of
the (weak) upper bound $\dim(P)\le 8$ will be complete once 
we have verified the following claim.

\medskip
\noindent
\textbf{Claim 2.}\quad
Each of the sets in the family $\{\cgS_1, \cgS_2, \cgS_3,\cgS_4,\cgS_5\}$ 
is reversible.  Furthermore, the set $\cgS_6$ can be covered by
two reversible sets.

\begin{proof} We will examine one set at a time, grouping sets with symmetric arguments.

\smallskip
\noindent
\textbf{Case $\cgS_1$ ($\cgS_2$).} We will first 
give a proof by contradiction to show that $\cgS_1$ is reversible.
The argument for $\cgS_2$ is symmetric.  Suppose to the contrary
that $\cgS_1$ is not reversible.  Let $S=\{(x_i,y_i):
1\le i\le k\}$ be  a strict alternating cycle contained
in $\cgS_1$.  For each $i\in[k]$, let
$a_i=a(y_i)$, the least element of $M[y_i]$ in the linear order $L$.

For each $i\in [k]$, since
$(x_i,y_i)\in\cgS_1$, we know that $x_i<a_i$ in $L$.
On the other hand, we know that $x_i\le y_{i+1}$ in $P$.
Therefore $a_{i+1}\le x_i$ in $L$.  In turn, this
implies $a_{i+1}<a_i$ in $L$.  Clearly, this statement cannot
hold for all $i\in[k]$.  The 
contradiction completes the proof for this part of the
claim.

\smallskip
\noindent
\textbf{Case $\cgS_3$.}
Now we give a proof by contradiction to show that the
set $\cgS_3$ of Type~$1B$ pairs is reversible. 
This argument will be more substantive
than the preceding case.  Suppose that $S=\{(x_i,y_i):1\le i\le k\}$
is a strict alternating cycle of pairs from $\cgS_3$.  For each 
$i\in[k]$, we let $a_i=a(y_i)$, $b_i=b(y_i)$, $s_i=s(y_i)$
and $t_i=t(y_i)$.  Since $(x_i,y_i)\in \cgS_3$, we know \[a_i=s_i<x_i<
t_i=b_i \quad \text{in }L.\]  Furthermore, we know $a_i$ is left of $b_i$ in $D_P[y_i]$.  
Let $z_i=z(y_i,a_i,b_i)$, $\cgE_i=\cgE[a_i,b_i]$,
$\cgC_i=\cgC(y_i,a_i,b_i)$ and $\cgR_i= \cgR(y_i,a_i,b_i)$

Now let $i\in [k]$ be arbitrary.
Since $a_i<x_i<b_i$ in $L$, we know $x_i$ is on the
path $\cgE_i$.
Since $S$ is a strict alternating cycle, we know $x_i<y_{i+1}$ 
in $P$ and $y_i\parallel y_{i+1}$ in $P$.  Let $W[x_i,y_{i+1}]$ be an 
arbitrary witnessing path. Clearly, $y_{i+1}$ is not
a minimal element in $P$, so all points
in the plane on the witnessing path $W[x_i,y_{i+1}]$ \emph{except}
$x_i$ are in the interior of $\cgC_i$.

Next we consider the set $M[y_{i+1}]$ which includes $x_i$.
We assert that all elements of 
$W[y_{i+1}]$ are on $\cgE_i$. To see this, suppose $u\in M[y_{i+1}]$, 
and $u$ is not on $\cgE_i$.  Then $u$ is in the exterior of $\cgC_i$.  
Let $W[u,y_{i+1}]$ be an arbitrary witnessing path.  Then this 
path must intersect the boundary of $\cgC_i$, and this forces $u<y_i$ 
in $P$, which is false.  The contradiction confirms our assertion.  

We conclude that:

\begin{equation}\label{eqn:halting}
a_i\le a_{i+1}\text{\quad and \quad} b_{i+1}\le b_i\text{\quad in $L$}.
\end{equation}

Of course, we also know that $a_{i+1}\le b_i$ in $L$, but we elect
to write the two inequalities in~\eqref{eqn:halting} in a weak form.
Since $i\in [k]$ was arbitrary, these inequalities hold for all $i\in[k]$.
We conclude that are minimal elements $a_0$ and $b_0$ so
that $a_i=a_0$ and $b_i=b_0$ for each $i\in[k]$.
The rules for determining  $z_i$ and the witnessing paths
$W[a_0,z_i]$ and $W[b_0,z]$ force
$\cgR_{i+1}$ to be a \emph{proper} subset of $\cgR_i$.  Clearly, this
is a contradiction since the strict set inclusion statement cannot
hold for all $i\in[k]$.  This completes the proof that
the set $\cgS_3$ consisting of all Type~$1B$ pairs is reversible.

\smallskip
\noindent
\textbf{Case $\cgS_4$ ($\cgS_5$).}
Next, we argue by contradiction that the set $\cgS_4$ of all
Type~$2B$ pairs is reversible.  The argument for the set
$\cgS_5$ of all Type~$2D$ pairs is symmetric. Suppose to the
contrary that $\cgS_4$ is not reversible, and let
$S=\{(x_i,y_i):1\le i\le k\}$ be a strict alternating cycle
contained in $\cgS_4$.  We use the same abbreviations as
in the preceding case for $a_i$, $b_i$, $s_i$ and $t_i$.
Since $(x_i,y_i)$ is Type~$2B$, we know $a_i<x_i<t_i<s_i\le b_i$ in $L$.

For each $i\in[k]$, we set $z_i=z(y_i,a_i,t_i)$,
$\cgE_i=\cgE[a_i,t_i]$, $\cgC_i=\cgC(y_i,a_i,t_i)$ and
$\cgR_i=\cgR(y_i,a_i,t_i)$.  It follows that $y_{i+1}$ is
in the interior of $\cgR_i$.  We now assert that
all points of $M[y_{i+1}]$ come from $\cgE[a_i,b_i]$.  To
see this, let $u$ be an element of $M[y_{i+1}]$ which does
not belong to $\cgE(a_i,b_i)$.  Then $u$ is in the exterior
of $\cgR_i$, and a witnessing path $W[u,y_{i+1}]$ would
have to intersect $\cgC_i$.  This forces $u<y_i$ in $P$ so
that $u\in\cgE(a_i,b_i)$, as desired.  In turn, this implies
that inequality~\eqref{eqn:halting} holds.  Since this inequality
holds for all $i\in[k]$, we know there
are elements $a_0,b_0\in M$ so that
$a_i=a_0$ and $b_i=b_0$ for all $i\in[k]$.  

Now we assert that $t_{i+1}\le t_i$ for all $i\in[k]$.
To the contrary, suppose that $t_{i+1}$ does not belong
to $\cgE[a_0,t_i]$.  Let $W[t_{i+1},y_{i+1}]$ be any
witnessing path.  Then this path intersects $\cgC_i$.
Let $v$ be the unique element of $P$ which is lowest in the
plane, and is common to $W[t_{i+1},y_{i+1}]$ and the  boundary
of $\cgC_i$.  Clearly, $v<z_i$ in $P$.  If $v$ is on
$W[a_0,z_i]$, we conclude that $a_0<t_{i+1}<t_i$ in the
left-to-right order in $D_P[y_i]$. This would imply that
$t_{i+1}<t_i$ in $L$ which is false.  We are left to
conclude that $v\in W[t_i,z_i]$ so that $a_0<t_i<t_{i+1}$
in $D_P[y_i]$, which contradicts the definition of $t_i$.
We conclude that our assertion that $t_{i+1}\le t_i$ in $L$
is correct.  Since $S$ is a strict alternating cycle, we 
know that there is a point $t_0\in M$ so that $t_i=t_0$ 
for all $i\in[k]$.

Now the same argument used in proving that the set $\cgS_3$
of all Type~$1B$ pairs is reversible shows that region
$\cgR_{i+1}$ is a proper subset of $\cgR_i$.  Clearly,
this statement cannot hold for all $i\in[k]$, and
this completes the proof that the set $\cgS_4$ consisting
of all Type~$2B$ pairs is reversible.

\smallskip
\noindent
\textbf{Case $\cgS_6$.} Now we turn to the last statement of Claim~$2$ where we must 
prove that the set $\cgS_6$ of all Type~$2C$ pairs can be covered by
two reversible sets.  Note that two Type~$2C$ pairs
in Figure~\ref{fig:top} shows that $\cgS_6$ may not
be reversible.

Let $(x,y)$ be a Type~$2C$ pair, and let $a=a(y)$, $b=b(y)$
and $z=z(y,a,b)$.  We will say that
$(x,y)$ is \textit{left-biased} if there
is a Type~$2$ element $y'\in P$ such that (1$'$)~$a(y')=a$ and
$b(y')=b$; (2$'$)~$z(y',a,b) = z$; and (3$'$)~$x$ is left of $b(y)$ in
$D_P[y']$.  Similarly, we will say that $(x,y)$ is 
\textit{right-biased} if there is an a Type~$2$ element
$y''$ satisfying (1$''$)~$a(y'')=a$ and
$b(y'')=b$; (2$''$)~$z(y'',a,b) = z$; and also (3$''$)~$x$ is right of $a$
in $D_P[y'']$.

We assert that there is no Type~$2C$ pair $(x,y)$ which is
both left-biased and right-biased.  If this were to happen,
we observe that $a,b,z$ belong to both $D_P[y']$ and 
$D_P[y'']$.  We would require that $x<a<b$ in the left-to-right
order on $D_P[y']$ and $a<b<x$ in the left-to-right order in
$D_P[y'']$.  In particular, both the pairs $(a,x)$ and
$(b,x)$ violate Proposition~\ref{pro:C2}.  This proves
that the assertion is correct.

We now show that the set $\cgS'$ of all Type~$2C$ pairs
which are \emph{not} right-biased is reversible.  The argument
to show that the set $\cgS''$ of all Type~$2C$ pairs which
are \emph{not} left-biased is symmetric.  Once this has been
accomplished, the proof that the set $\cgS_6$ consisting
of all Type~$2C$ pairs can be covered by two reversible
sets will be complete.

We argue by contradiction and let $S=\{(x_i,y_i):1\le i\le k\}$
be a strict alternating cycle of Type~$2C$ pairs, none of which
are left-biased.  For each $i\in [k]$, we use the now standard 
abbreviations $a_i,b_i,s_i,t_i$. We then take $z_i=z(y_i,a_i,b_i)$,
$\cgE_i=\cgE[a_i,b_i]$, $\cgC_i=\cgC(y_i,a_i,b_i)$ and
$\cgR_i=\cgR(y_i,a_i,b_i)$.

Arguments just like those applied earlier show that 
inequality~\eqref{eqn:halting} holds.  Therefore,
there are elements $a_0,b_0\in M$
such that $a_i=b_0$ and $b_i=b_0$ for all $i\in[k]$.

Now let $i\in [k]$ be arbitrary.  
We then observe that
$\cgR_{i+1}$ is a proper subset of $\cgR_i$ \emph{unless}
$z_{i+1}=z_i$.  In this case, $\cgR_{i+1}=\cgR_i$.  It
follows that there is an element $z_0\in P$ and a simple
closed curve $\cgC_0$ enclosing a region $\cgR_0$ so that
$z_i=z_0$, $\cgC_i=\cgC_0$ and $\cgR_i=\cgR_0$ for all
$i\in[k]$.

After a relabeling if necessary, we may assume that $s_1\le s_i$
in $L$ for each $i\in[k]$.   Then
$t_1<x_1<s_1$ in $L$. Since $x_1<y_2$ in $P$, either $s_2 <x_1<b_0$ in $L$ or $a_0<x_1<t_2$ in $L$.  If $s_2\le x_1<b_0$ 
in $L$, then $s_2<s_1$ in $L$ which is false.  We conclude that
$a_0<x_1\le t_2$ in $L$.  Therefore, $x_1$ is right of $a_0$
in $D_P[y_2]$.  This shows that $(x_1,y_1)$ is right-biased.
The contradiction completes the proof.
\end{proof}

As promised, we now show how to improve Claim~$2$ by
showing the set $\cgS_0$ can be covered by~$5$ reversible
sets.  This will be accomplished by proving the following
two claims.

\medskip
\noindent
\textbf{Claim 3.}\quad
The set $\cgS_3\cup\cgS_4$ of all pairs which are either
Type~$2B$ or Type~$2D$ is reversible.

\medskip
\noindent
\textbf{Claim 4.}\quad
The set $\cgS_7$ consisting of all pairs which are
either Type~$1B$ or Type~$2C$ but not right-biased
is reversible.

\begin{proof}
We first show by contradiction that $\cgS_3\cup\cgS_4$
is reversible.  Let $S=\{(x_i,y_i):1\le i\le k\}$ be
a strict alternating cycle of pairs from $\cgS_3\cup\cgS_4$
In view of our earlier arguments, there must be at
least one pair in $S$ of Type~$2B$ and at least one
pair of Type~$2D$.

The abbreviations $a_i,b_i,s_i,,t_i$ are just
as before.  Now we know that $a_i\le t_i<s_i\le b_i$ in $L$.
Furthermore, if $(x_i,y_i)$ is Type~$2B$, we know
$a_i<x_i<t_i$ in $L$, and if $(x_i,y_i)$ is Type~$2D$, we know
$s_i<x_i<b_i$ in $L$.

Now let $i\in [k]$. If $a_i=t_i$, we set $z_i=a_i$, and
we let $\cgR_i$ be the region in the plane consisting only
of the point $a_i$. If $a_i<t_i$ in $L$,
we set $z_i=z(y_i,a_i,t_i)$, 
$\cgC_i=\cgC(y_i,a_i,b_i)$ and 
$\cgR_i=\cgR(y_i,a_i,t_i)$.  Analogously, 
if $s_i=t_i$, we set $v_i=b_i$ and we take $\cgT_i$ as
the region in the plane consisting only of the point
$b_i$.
If $s_i<b_i$ in $L$, we
let $v_i=z(y_i,s_i,b_i)$, $\cgD_i=
\cgC(y_i,s_i,b_i)$ and $\cgT_i=\cgR(y_i,s_i,t_i)$,

Repeating arguments already presented, we quickly
learn that there are elements $a_0,b_0,s_0,t_0\in M$, 
elements $z_0,v_0\in P$, simple closed curves $\cgC_0$
and $\cgD_0$ and regions $\cgR_0$, $\cgT_0$ so that
$a_i=a_0$, $b_i=b_0$, $z_i=z_0$, $v_i=v_0$,
$\cgC_i=\cgC_0$, $\cgR_i=\cgR_0$, $\cgD_i=\cgD_0$, and
$\cgT_i=\cgT_0$ for all $i\in[k]$.  

If $i\in [k]$ and $(x_i,y_i)$ is Type~$2B$, then
it is easy to see that $\cgR_{i+1}\subsetneq \cgR_i$ while
$\cgT_{i+1}=\cgT_i$.  Analogously, if $(x_i,y_i)$ is
Type~$2D$ then $\cgT_{i+1}\subsetneq \cgT_i$ while
$\cgR_{i+1}=\cgR_i$.  Clearly, these statements result in
a contradiction, so we have completed the proof that
$\cgS_3\cup\cgS_4$ is reversible.

Now we prove by contradiction  that $\cgS_7$, which
consists of all Type~$1B$ pairs and all Type~$2C$ pairs
which are not right-biased is reversible.  Let
$S=\{(x_i,y_i):1\le i\le k\}$ be a strict alternating
cycle contained in $\cgS_7$.  Then we know that
$S$ contains both a Type~$1B$ pair and a Type~$2C$ pair.

Now suppose that $i\in [k]$.  We set
$a_i=a(y_i)$, $b_i=b(y_i)$, $z_i=z(y_i,a_i,b_i)$,
$\cgC_i=\cgC(y_i,z_i,b_i)$ and $\cgR_i=\cgR(y_i,a_i,b_i)$.
Then $y_{i+1}$ is in the interior of $\cgR_{i}$
and all elements of $M[y_{i+1}]$ are on the path
$\cgE(a_i,b_i)$.  It follows that the inequalities in~\eqref{eqn:halting}
hold.  We conclude that there are elements $a_0,b_0\in M$
such that $a_i=a_0$ and $b_i=b_0$ for all $i\in[k]$.

Let $i$ and $j$ be integers in $[k]$ so
that $(x_i,y_i)$ is Type~$1B$ and $(x_j,y_j)$ is Type~$2C$.
Then $a_0$ is left of $b_0$ in $D_P[y_i]$ and
$a_0$ is right of $b_0$ in $D_P[y_j]$.  These statements
contradict Proposition~\ref{pro:C2}.  With these observations,
the proof of Claim~$3$ is complete.  This also completes the
proof of our main theorem.
\end{proof}
\end{proof}

\section{Closing Comments and Open Problems}\label{sec:close}

We pause to explain our motivation in studying the class of
$\AFB$-posets.  Let $P$ be a planar poset and let
$x_0$ be an arbitrary minimal element of $P$.  Then set
$A_0=\{x_0\}$ and  let $B_0$ consist of all elements $y$ in $P$
such that $y>x_0$ in $P$.
If $i\ge0$ and we have defined a sequence
$(A_0,B_0,A_1,B_1,\dots,A_i,B_i)$ of pairwise disjoint
subposets of $P$ and their union is a proper connected subposet 
$Q$ of $P$, we let $A_{i+1}$ consist of all elements $x\in P-Q$ for which
there is some $y\in B_i$ such that $x<y$ in $P$.  Also, when
$Q\cup A_{i+1}$ is a proper subposet of $P$, we take
$B_{i+1}$ as the set of all $y\in P-(Q\cup A_{i+1})$ for
which there is some $x\in A_{i+1}$ for which $x<y$ in $P$.

The resulting partition of $P$ is now known as an \textit{unfolding}
of $P$, and this concept has been used in several papers,
including~\cite{bib:StrTro}, \cite{bib:MicWie} 
and~\cite{bib:JoMiWi-1}.  The key feature for our purposes is 
that for all $i\ge0$, the subposet $B_i$ is an $\AFB$-poset, and the 
dual of the subposet $A_i$ is an $\AFB$-poset.
As is well known there is some $i\ge0$ for which:
\[
\max\{\dim(A_i\cup B_i),\dim(B_i\cup A_{i+1})\}\ge\dim(P)/2.
\]
It follows that when the dimension of $P$ is \emph{very} large,
we now know that there is a subposet of $P$ in which
the difficulty of the dimension problem has a ``bipartite flavor,''
i.e., the poset is the union of two relatively simple subposets,
one a down set and the other an up set and the challenge is to
reverse incomparable pairs of the form $(x,y)$ where $x$ is in the
down set and $y$ is in the up set. Reversing  the remaining incomparable
pairs takes at most~$6$ linear extensions.

Our motivation for this line of research has from the outset
been to develop machinery for attacking the following long-standing and
apparently quite challenging conjectures:

\begin{conjecture}\label{con:planar-Sd}
A planar poset with large dimension contains a large standard
example, i.e., for every $d\ge2$, there exists a constant $d_0$ so
that if $P$ is a planar poset and $\dim(P)\ge d_0$, then $P$
contains the standard example $S_d$ as a subposet.
\end{conjecture}  

We believe, but cannot be certain, that the first reference to this
conjecture is on page~$119$ in~\cite{bib:Trot-Book}, as it
has become part of the folklore of the subject.

In fact, probably the following considerably stronger conjecture
is true.

\begin{conjecture}\label{con:genus-Sd}
For every pair $(n,d)$ of positive integers with $d\ge2$, there
is an integer $d_0$ so that if $P$ is a poset and $\dim(P)\ge d_0$, then
either $P$ contains the standard example $S_d$ or the cover graph of
$P$ contains a $K_n$ minor.
\end{conjecture}  

In just the last two years, there has been considerable interest
in two variants of the original Dushnik-Miller notion of dimension.
They are called \textit{Boolean dimension} and \textit{local dimension}. We refer readers to~\cite{bib:TroWal}, \cite{bib:FeMeMi},
\cite{bib:MeMiTr} and~\cite{bib:BoGrTr} for definitions and results.

Specific to our interests here is the proof by Bosek, Grytczuk and
Trotter~\cite{bib:BoGrTr} that local dimension is not bounded for
planar posets.  The following conjecture is due to Ne\v{s}et\v{r}il and
Pudlak and is given in question form in their 1989 paper~\cite{bib:NesPud}
in which the concept of Boolean dimension is first introduced.
 
\begin{conjecture}\label{con:planar-bd}
The Boolean dimension of planar posets is bounded, i.e., there
is a constant $d_0$ so that if $P$ is a planar poset,
then the Boolean dimension of $P$ is at most $d_0$.
\end{conjecture}

We believe that the results presented here will prove useful in
attacking this conjecture with the assistance of the concept
of unfolding.

\end{document}